\documentclass[reqno,11pt]{article} 
\usepackage{amsmath, amssymb}
\usepackage{amsthm} 
\theoremstyle{plain} 
\newtheorem{theorem}{\indent\sc Theorem}[section]
\newtheorem{lemma}[theorem]{\indent\sc Lemma}
\newtheorem{corollary}[theorem]{\indent\sc Corollary}
\newtheorem{proposition}[theorem]{\indent\sc Proposition}

\theoremstyle{definition} 

\newtheorem{remark}[theorem]{\indent\sc Remark}
\newtheorem{example}[theorem]{\indent\sc Example}

\makeatletter
\def\address#1#2{\begingroup
\noindent\parbox[t]{7.8cm}{%
\small{\scshape\ignorespaces#1}\par\vskip1ex
\noindent\small{\itshape E-mail address}%
\/: #2\par\vskip4ex}\hfill%
\endgroup}%
\makeatother
\title{Some applications of modular units} 
\author{
\textsc{Ick Sun Eum, Ja Kyung Koo and Dong Hwa Shin$^*$} 
}
\date{} 
\begin{document}

\maketitle

\footnote{
\thanks{
The first and second named authors were partially supported by the NRF of Korea
Grant funded by MEST (2012-0000798). $^*$The corresponding author was supported
by the NRF of Korea Grant funded by the Korean Government (2012R1A1A1013132).} }

\begin{abstract}
We show that a weakly holomorphic modular function can be written as a sum of modular units of higher level.
We further find a necessary and sufficient condition for a Siegel modular function of degree $g$
to have
neither zero nor pole on the domain when restricted to certain subset of the Siegel upper half-space $\mathbb{H}_g$.
\end{abstract}

\section {Introduction}

Let $g$ be a positive integer. We let
\begin{equation*}
\mathbb{H}_g =\{Z\in
\mathrm{Mat}_g(\mathbb{C})~|~{^t}Z=Z,~\mathrm{Im}(Z)~\textrm{is
positive definite}\}
\end{equation*}
be the Siegel upper half-space of degree $g$
on which the symplectic group
\begin{equation*}
\mathrm{Sp}_{g}(\mathbb{Z})=\{\gamma\in\mathrm{GL}_{2g}(\mathbb{Z})~|~ {^t}\gamma
J\gamma=J \}\quad\textrm{with}~J=\left[\begin{matrix}0 & -I_g\\I_g & 0
\end{matrix}\right]
\end{equation*}
acts by the rule
\begin{equation*}
\left[\begin{matrix}A & B\\C&D\end{matrix}\right](Z)
=(AZ+B)(CZ+D)^{-1},
\end{equation*}
where $A,B,C,D$ are $g\times g$ block matrices. For a positive integer $N$ we further let
\begin{equation*}
\Gamma(N)=\{\gamma\in\mathrm{Sp}_{g}(\mathbb{Z})~|~ \gamma\equiv
I_{2g}\pmod{N}\}
\end{equation*}
be the congruence subgroup modulo $N$
of the group $\mathrm{Sp}_{g}(\mathbb{Z})$.
In particular, when $g=1$, $\mathbb{H}_g$ becomes the upper half-plane $\mathbb{H}=\{
\tau\in\mathbb{C}~|~\mathrm{Im}(\tau)>0\}$
 and $\mathrm{Sp}_{g}(\mathbb{Z})=\mathrm{SL}_2(\mathbb{Z})$ acts on it by fractional
 linear transformations.
\par
Define a subset $\mathbb{H}_g^\mathrm{diag}$ of $\mathbb{H}_g$ by
\begin{equation*}
\mathbb{H}_g^\mathrm{diag}=\{\mathrm{diag}(\tau_1,\tau_2,\ldots,\tau_g)~|~\tau_1,\tau_2,\ldots,\tau_g\in\mathbb{H}\},
\end{equation*}
where
$\mathrm{diag}(\tau_1,\tau_2,\ldots,\tau_g)$ stands for the $g\times g$ diagonal matrix whose diagonal entries are
$\tau_1,\tau_2,\ldots,\tau_g$. If $g=1$, then $\mathbb{H}_g^\mathrm{diag}$ is nothing but $\mathbb{H}$.
Let $f(Z)$ be a (meromorphic) Siegel modular function of degree $g$ and level $N$ (over $\mathbb{C}$),
namely $f(Z)$ is a quotient of two Siegel modular forms of degree $g$ and the same weight so that it is invariant
under $\Gamma(N)$.
When $g=1$, $f$ becomes a usual meromorphic modular function of level $N$.
We shall mainly consider the case where $f$ has neither zero nor pole on $\mathbb{H}_g^\mathrm{diag}$.
\par
Let $X(N)=\overline{\Gamma}(N)\backslash\mathbb{H}^*$ be the
modular curve of level $N$ that is a compact Riemann surface, where
$\overline{\Gamma}(N)=\Gamma(N)/\{\pm I_2\}$ and
$\mathbb{H}^*=\mathbb{H}\cup\mathbb{Q}\cup\{i\infty\}$. We denote its function field by $\mathbb{C}(X(N))$. As is well-known,
$X(1)$ is of genus zero and
$\mathbb{C}(X(1))=\mathbb{C}(j)$, where
\begin{equation*}
j=j(\tau)=q^{-1}+744+196884q+21493760q^2
+864299970q^3+\cdots\quad(q=e^{2\pi i\tau},~i=\sqrt{-1})
\end{equation*}
is the elliptic modular function \cite[Theorem 2.9]{Shimura}.
Furthermore, $\mathbb{C}(X(N))$ is a Galois extension of $\mathbb{C}(X(1))$
whose Galois group is naturally isomorphic to
$\overline{\Gamma}(1)/\overline{\Gamma}(N)$. Let
$\mathcal{O}_N$ be the integral closure of $\mathbb{C}[j]$ in
$\mathbb{C}(X(N))$. We call the invertible elements in $\mathcal{O}_N$
\textit{modular units} of level $N$ (over $\mathbb{C}$), which are precisely those
functions in $\mathbb{C}(X(N))$ having no zeros and poles on
$\mathbb{H}$ \cite[p.36]{K-L}. Kubert and Lang developed in
\cite{K-L} the theory of modular units in terms of Siegel functions
which will be defined in $\S$\ref{sectwo}. (In addition, they require that the Fourier coefficients
of a modular unit of level $N$ lie in the $N$th cyclotomic field.)
In this paper we shall first describe $\mathcal{O}_N$ in view of
modular units when $N\equiv0\pmod{4}$ (Theorem \ref{closure}),
and then conclude that any weakly holomorphic modular function can be expressed as
 a sum of modular units of higher level (Corollary \ref{sum}). Here, a function is said to be
 weakly holomorphic if it is holomorphic on $\mathbb{H}$.
 \par
On the other hand, suppose that
$g$ and $N$ are two positive integers $\geq2$, and
 let $f(Z)$ be a Siegel modular function of degree $g$ and level $N$.
We shall further prove that $f(Z)$ has neither zero nor pole on $\mathbb{H}_g^\mathrm{diag}$ if and only if
$f(\mathrm{diag}(\tau_1,\tau_2,\ldots,\tau_g))$ is a product of $g$ modular units of variables
$\tau_1,\tau_2,\ldots,\tau_g\in\mathbb{H}$, respectively
(Theorem \ref{prod2}). To this end we shall examine some necessary basic properties of modular units in $\S$\ref{sectwo}.
And, we shall show that certain quotient of theta constants of degree $g$ on
$\mathbb{H}_g^\mathrm{diag}$ is a product of modular units (Example \ref{thetaconstant}).

\section {Properties of modular units}\label{sectwo}

For a positive integer $N$ we denote the group of all modular units of level $N$ by $V_N$ (that is, $V_N=\mathcal{O}_N^\times$), which contains $\mathbb{C}^\times$ as a subgroup.
In this section we shall develop some necessary properties about modular units which will be used in later sections.

\begin{lemma}\label{holj}
If $f$ is a weakly holomorphic modular function of level $1$, then it
is a polynomial of $j$ over $\mathbb{C}$, that is $f\in\mathbb{C}[j]$.
\end{lemma}
\begin{proof}
\cite[Theorem 2]{Lang}.
\end{proof}

\begin{remark}
Note that
$j$ gives rise to a bijection $j:\overline{\Gamma}(1)\backslash\mathbb{H}\rightarrow
\mathbb{C}$ \cite[Chapter 3 Theorem 4]{Lang}.
\end{remark}

\begin{proposition}\label{characterize}
Let $h\in\mathbb{C}(X(N))$. Then, $h$ is weakly holomorphic
 if and only if $h$ is integral over $\mathbb{C}[j]$.
\end{proposition}
\begin{proof}
Assume that $h=h(\tau)$ is weakly holomorphic. We consider the following monic polynomial of $X$
\begin{equation*}
P(X)=\prod_{\gamma\in\overline{\Gamma}(1)/\overline{\Gamma}(N)}(X-h\circ\gamma).
\end{equation*}
Since $\mathrm{Gal}(\mathbb{C}(X(N))/\mathbb{C}(X(1)))\simeq
\overline{\Gamma}(1)/\overline{\Gamma}(N)$,
every coefficient of $P(X)$ belongs to $\mathbb{C}(X(1))$ and is holomorphic on $\mathbb{H}$.
So, it is a polynomial of $j$ over $\mathbb{C}$ by Lemma \ref{holj}. This shows that
$h$ is integral over $\mathbb{C}[j]$.
\par
Conversely, assume that $h$ is integral over $\mathbb{C}[j]$. Then $h$ is a zero of a monic polynomial
\begin{equation*}
X^n+P_{n-1}(j)X^{n-1}+\cdots+P_1(j)X+P_0(j),
\end{equation*}
where $n\geq1$ and $P_{n-1}(X),\ldots,P_1(X),P_0(X)\in\mathbb{C}[j][X]$. Suppose on the contrary that $h$ has a pole
at $\tau_0\in\mathbb{H}$ (so, $h\neq0$). Since $h$ satisfies
\begin{equation*}
h^n+P_{n-1}(j)h^{n-1}+\cdots+P_1(j)h+P_0(j)=0,
\end{equation*}
we get by dividing both sides by $h^n$ and substituting $\tau=\tau_0$
\begin{equation*}
1+P_{n-1}(j(\tau_0))(1/h(\tau_0))+\cdots+P_1(j(\tau_0))(1/h(\tau_0))^{n-1}+P_0(j(\tau_0))(1/h(\tau_0))^n=0.
\end{equation*}
This yields a contradiction $1=0$ because $j(\tau_0)\in\mathbb{C}$ and $1/h(\tau_0)=0$.
Therefore $h$ must be weakly holomorphic.
\end{proof}

\begin{remark}
By definition, $h\in\mathbb{C}(X(N))$ is a modular unit if and only if
both $h$ and $h^{-1}$ are integral over $\mathbb{C}[j]$. Hence, Proposition \ref{characterize}
gives an elementary proof of the well-known fact that
$h$ is a modular unit if and only if it has no zeros and poles on $\mathbb{H}$ \cite[p.36]{K-L}.
\end{remark}

Given a vector
$\left[\begin{matrix}r\\s\end{matrix}\right]\in(1/N)\mathbb{Z}^2-\mathbb{Z}^2$ for $N\geq2$,
the Siegel function
$g_{\left[\begin{smallmatrix}r\\s\end{smallmatrix}\right]}(\tau)$ is
defined on $\mathbb{H}$ by the following infinite product
\begin{equation}\label{Siegel}
g_{\left[\begin{smallmatrix}r\\s\end{smallmatrix}\right]}(\tau)=-q^{(1/2)
(r^2-r+1/6)}e^{\pi is(r-1)}(1-q^re^{2\pi
is})\prod_{n=1}^{\infty}(1-q^{n+r}e^{2\pi is})(1-q^{n-r}e^{-2\pi
is}),
\end{equation}
which is a weakly holomorphic modular function of level $12N^2$ \cite[Chapter 3 Theorem 5.2]{K-L}.

\begin{lemma}\label{ranklemma}
Suppose $N\geq2$ and let $n$ be the number of inequivalent cusps of $X(N)$.
Then, the rank of the subgroup of $V_N/\mathbb{C}^\times$ generated by $g_{\left[\begin{smallmatrix}r\\s\end{smallmatrix}\right]}(\tau)^{12N}$ for
$\left[\begin{matrix}r\\s\end{matrix}\right]\in(1/N)\mathbb{Z}^2-\mathbb{Z}^2$ is $n-1$.
\end{lemma}
\begin{proof}
\cite[Chapter 2 Theorem 3.1]{K-L}.
\end{proof}

\begin{remark}
We have the formula
\begin{equation*}
n=|\overline{\Gamma}(1)/\overline{\Gamma}(N)|/N=\left\{\begin{array}{ll}3 & \textrm{if}~N=2,\\
(N^2/2)\prod_{p|N}(1-p^{-2}) & \textrm{if}~N>2
\end{array}\right.
\end{equation*}
\cite[pp.22--23]{Shimura}.
\end{remark}

\begin{proposition}\label{rank}
With the same assumption and notation as in \textup{Lemma \ref{ranklemma}}, $V_N/\mathbb{C}^\times$ is a free abelian group of rank $n-1$.
\end{proposition}
\begin{proof}
Let $\infty_1,\infty_2,\ldots,\infty_n$ be the inequivalent cusps of $X(N)$, and let $\mathcal{D}_N$ be the free abelian
group of rank $n$ generated by these cusps. Then, an element of $\mathcal{D}_N$ is uniquely written as
\begin{equation*}
m_1(\infty_1)+m_2(\infty_2)+\cdots+m_n(\infty_n)
\quad\textrm{for some integers}~m_1,m_2,\ldots,m_n.
\end{equation*}
Now, we consider a (well-defined) injective homomorphism
\begin{eqnarray*}
V_N/\mathbb{C}^\times&\rightarrow&\mathcal{D}_N\\
h&\mapsto&\mathrm{div}(h).
\end{eqnarray*}
If $h\in V_N/\mathbb{C}^\times$ has $\mathrm{div}(h)=\sum_{k=1}^n m_k(\infty_k)$,
then we get the relation $\sum_{k=1}^n m_k=0$.
Hence $V_N/\mathbb{C}^\times$ is a free abelian group of rank $\leq n-1$.
Thus it follows from Lemma \ref{ranklemma} that the rank of $V_N/\mathbb{C}^\times$ is exactly $n-1$.
\end{proof}

\begin{remark}\label{V_1}
Since every cusp of
$X(1)$ is equivalent to $i\infty$ \cite[p.14]{Shimura}, if $h\in V_1$ then
$\mathrm{div}(h)=m(i\infty)$ for some integer $m$.
On the other hand, now that the sum of the orders of zeros and poles of $h$ is zero, we get $m=0$.
This yields $V_1=\mathbb{C}^\times$.
\end{remark}

\begin{lemma}\label{constsurj}
Let $N\geq2$ and $h\in V_N-\mathbb{C}^\times$. There is a finite subset $S$ of $\mathbb{C}^\times$ so that
the map
\begin{eqnarray*}
\varphi~:~\mathbb{H}&\rightarrow&\mathbb{C}^\times-S\\
\tau&\mapsto&h(\tau)\nonumber
\end{eqnarray*}
is surjective.
\end{lemma}
\begin{proof}
Consider the following holomorphic map between compact Riemann surfaces
\begin{eqnarray*}
X(N)&\rightarrow&\mathbb{P}^1(\mathbb{C})\\
\tau&\mapsto&[h(\tau):1].
\end{eqnarray*}
Since $h$ is not a constant, the above map is surjective.
Take a subset
$S$ of $\mathbb{C}^\times$ as
\begin{equation*}
S=\{h(\tau)~|~\tau~\textrm{is a cusp of}~X(N)\}-\{0,\infty,h(\tau)~|~\tau\in\mathbb{H}\}.
\end{equation*}
Since there are only finitely many inequivalent cusps of $X(N)$, it is a finite set. And,
the map $\varphi$ becomes surjective.
\end{proof}

\begin{proposition}\label{notunit}
Let $h$ be a modular unit. Suppose that
\begin{equation}\label{orderassumption}
\mathrm{ord}_q~h\circ\gamma\neq0\quad\textrm{for all}~\gamma\in\mathrm{SL}_2(\mathbb{Z}).
\end{equation}
Then $h-c$  is not a modular unit for any $c\in\mathbb{C}^\times$.
\end{proposition}
\begin{proof}
Let us consider the holomorphic map between two compact Riemann surfaces
\begin{eqnarray*}
\varphi~:~X(N)&\rightarrow&\mathbb{P}^1(\mathbb{C})\\
\tau&\mapsto&[h(\tau):1].
\end{eqnarray*}
Since $h$ is not a constant by (\ref{orderassumption}), $\varphi$ is surjective.
\par
Now, let $c\in\mathbb{C}^\times$. Since $\varphi$ is surjective and the values of $\varphi$ at the cusps of $X(N)$ are either
$[0:1]$ or $[\infty:1]=[1:0]$ by (\ref{orderassumption}), there exists $\tau_0\in\mathbb{H}$ such that
$\varphi(\tau_0)=[c:1]$. This implies that $h(\tau)-c$ has a zero at $\tau=\tau_0$,
and hence $h-c$ is not a modular unit.
\end{proof}

\begin{example}
Let $N\geq2$ and $\left[\begin{matrix}r\\s\end{matrix}\right]\in(1/N)\mathbb{Z}^2-\mathbb{Z}^2$. Consider the Siegel function
\begin{equation*}
h(\tau)=g_{\left[\begin{smallmatrix}r\\s\end{smallmatrix}\right]}(\tau)^{12N},
\end{equation*}
which is a modular unit of level $N$ by Lemma \ref{ranklemma}. Then we have the following properties:
\begin{itemize}
\item[(i)] $h\circ\gamma=g_{{}^t\gamma\left[\begin{smallmatrix}r\\s\end{smallmatrix}\right]}(\tau)^{12N}$
for any $\gamma\in\mathrm{SL}_2(\mathbb{Z})$ \cite[Chapter 2 Proposition 1.3]{K-L},
\item[(ii)] $\mathrm{ord}_q~h=6N\cdot\mathbf{B}_2(\langle r\rangle)$, where
$\mathbf{B}_2(x)=x^2-x+1/6$ is the second Bernoulli polynomial and $\langle x\rangle$ is the fractional part of $x$ such that
$0\leq\langle x\rangle<1$ for $x\in\mathbb{R}$ \cite[p.31]{K-L},
\item[(iii)] $\mathbf{B}_2(x)\neq0$ for all $x\in\mathbb{Q}$.
\end{itemize}
Thus $h$ satisfies the assumption (\ref{orderassumption}) in Proposition \ref{notunit}.
\end{example}

\begin{remark}
If $h$ does not satisfy the assumption (\ref{orderassumption}),
then
$h-c$ could be a modular unit for some constant $c\in\mathbb{C}^\times$ (see Remark \ref{explicit}).
\end{remark}

\section {Integral closures in modular function fields}

In this section, when $N\equiv0\pmod{4}$ we shall investigate
explicit generators of the integral closure $\mathcal{O}_N$ of $\mathbb{C}[j]$
in $\mathbb{C}(X(N))$ by using the Weierstrass units.
\par
For a lattice $L=[\omega_1,\omega_2]=\mathbb{Z}\omega_1+\mathbb{Z}\omega_2$ in $\mathbb{C}$
 the Weierstrass $\wp$-function is defined by
\begin{equation*}
\wp(z;L)=\frac{1}{z^2}+\sum_{\omega\in L-\{0\}}\bigg(\frac{1}{(z-\omega)^2}
-\frac{1}{\omega^2}\bigg)
\quad(z\in\mathbb{C}).
\end{equation*}

\begin{lemma}\label{wp}
Let $z,w\in\mathbb{C}-L$. Then, $\wp(z;L)=\wp(w;L)$ if and only if $z\equiv\pm w\pmod{L}$.
\end{lemma}
\begin{proof}
\cite[Chaper IV $\S$3]{Silverman}.
\end{proof}

Let $N\geq2$. For a vector $\left[\begin{matrix}r\\s\end{matrix}\right]\in(1/N)\mathbb{Z}^2-\mathbb{Z}^2$ we define
\begin{equation*}
\wp_{\left[\begin{smallmatrix}r\\s\end{smallmatrix}\right]}(\tau)=\wp(r\tau+s;[\tau,1])\quad(\tau\in\mathbb{H}),
\end{equation*}
which is a weakly holomorphic modular form of level $N$ and weight $2$ \cite[Chapter 6]{Lang}. More precisely,
it satisfies the transformation formula
\begin{equation}\label{wptransf}
\wp_{\left[\begin{smallmatrix}r\\s\end{smallmatrix}\right]}(\tau)\circ\gamma
=(c\tau+d)^2\wp_{{}^t\gamma\left[\begin{smallmatrix}r\\s\end{smallmatrix}\right]}(\tau)
\quad\textrm{for any}~\gamma
=\left[\begin{matrix}a&b\\c&d\end{matrix}\right]\in\mathrm{SL}_2(\mathbb{Z}).
\end{equation}
Hence
the following function
\begin{equation*}
(\wp_{\left[\begin{smallmatrix}a_1\\b_1\end{smallmatrix}\right]}(\tau)-
\wp_{\left[\begin{smallmatrix}c_1\\d_1\end{smallmatrix}\right]}(\tau))/
(\wp_{\left[\begin{smallmatrix}a_2\\b_2\end{smallmatrix}\right]}(\tau)-
\wp_{\left[\begin{smallmatrix}c_2\\d_2\end{smallmatrix}\right]}(\tau))
\end{equation*}
for $\left[\begin{matrix}a_k\\b_k\end{matrix}\right],
\left[\begin{matrix}c_k\\d_k\end{matrix}\right]\in(1/N)\mathbb{Z}^2-\mathbb{Z}^2$ with
$\left[\begin{matrix}a_k\\b_k\end{matrix}\right]\not\equiv
\pm\left[\begin{matrix}c_k\\d_k\end{matrix}\right]\pmod{\mathbb{Z}^2}$ ($k=1,2$)
is a modular unit of level $N$ by Lemma \ref{wp},
which is called a Weierstrass unit of level $N$.
\par
We further define three functions on $\mathbb{H}$
\begin{eqnarray*}
g_2(\tau)&=&60\sum_{\omega\in[\tau,1]-\{0\}}\omega^{-4},\\
g_3(\tau)&=&140\sum_{\omega\in[\tau,1]-\{0\}}\omega^{-6},\\
\Delta(\tau)&=&g_2(\tau)^3-27g_3(\tau)^2,
\end{eqnarray*}
which are modular forms of level $1$ and weight $4$, $6$ and $12$, respectively \cite[Chapter 3 Theorem 3]{Lang}.
\par
For a positive integer $N$, let
\begin{equation*}
\Gamma_1(N)=\{\gamma\in\mathrm{SL}_2(\mathbb{Z})~|~\gamma\equiv\left[\begin{matrix}1 & *\\0 &1
\end{matrix}\right]\pmod{N}\},
\end{equation*}
and let $X_1(N)=\overline{\Gamma}_1(N)\backslash\mathbb{H}^*$ be the corresponding modular curve
where $\overline{\Gamma}_1(N)=\Gamma_1(N)/\{\pm I_2\}$.

\begin{lemma}\label{generators}
\begin{itemize}
\item[\textup{(i)}]
If $N\geq2$, then
$\mathbb{C}(X_1(N))=\mathbb{C}(j,(g_2g_3/\Delta)\wp_{\left[\begin{smallmatrix}0\\1/N\end{smallmatrix}\right]})$.
\item[\textup{(ii)}] If $N\geq2$, then
$\mathbb{C}(X(N))=\mathbb{C}(X_1(N))((g_2g_3/\Delta)\wp_{\left[\begin{smallmatrix}1/N\\0\end{smallmatrix}\right]})$.
\item[\textup{(iii)}] $\mathbb{C}(X_1(4))=\mathbb{C}(g_{\left[\begin{smallmatrix}
1/4\\0\end{smallmatrix}\right]}(4\tau)^{-8}
g_{\left[\begin{smallmatrix}
1/2\\0\end{smallmatrix}\right]}(4\tau)^8)$.
\end{itemize}
\end{lemma}
\begin{proof}
(i), (ii) \cite[Proposition 7.5.1]{D-S}.\\
(iii) \cite[Table 2]{K-S}.\\
\end{proof}

The modular curve $X_1(4)$ is of genus $0$ and has three inequivalent cusps, namely
$0$, $1/2$ and $i\infty$ \cite[p.131]{K-K}. Put
\begin{equation*}
g_{1,4}(\tau)=g_{\left[\begin{smallmatrix}
1/4\\0\end{smallmatrix}\right]}(4\tau)^{-8}
g_{\left[\begin{smallmatrix}
1/2\\0\end{smallmatrix}\right]}(4\tau)^8,
\end{equation*}
which is a primitive generator of $\mathbb{C}(X_1(4))$ over $\mathbb{C}$ by Lemma \ref{generators}(iii).
It then follows from \cite[Theorem 6.5]{K-S} that the map
\begin{eqnarray*}
X_1(4)=\overline{\Gamma}_1(4)\backslash\mathbb{H}^*&\rightarrow&\mathbb{P}^1(\mathbb{C})\\
\tau&\mapsto&[g_{1,4}(\tau):1]
\end{eqnarray*}
is an isomorphism between compact Riemann surfaces. Moreover, $g_{1,4}(\tau)$ has values $16$, $0$, $\infty$ at the
cusps $\tau=0$, $1/2$, $i\infty$,
respectively (\cite[Theorem 3(ii)]{K-K} and \cite[Table3]{K-S}). Thus we claim that
\begin{equation}\label{observation}
\textrm{$g_{1,4}-c$ for $c\in\mathbb{C}$ is a modular unit (for $\Gamma_1(4)$) $\Longleftrightarrow$ $c=16$ or $0$}.
\end{equation}

\begin{theorem}\label{closure}
Let $\mathcal{O}_{1,N}$ and $\mathcal{O}_N$ be the integral closures of
$\mathbb{C}[j]$ in $\mathbb{C}(X_1(N))$ and $\mathbb{C}(X(N))$, respectively.
Assume that $N\equiv0\pmod{4}$.
\begin{itemize}
\item[\textup{(i)}] $\mathcal{O}_{1,4}=\mathbb{C}[g_{1,4},g_{1,4}^{-1},(g_{1,4}-16)^{-1}]$.
\item[\textup{(ii)}] $\mathcal{O}_{1,N}=\mathcal{O}_{1,4}[h_{1,N}]$, where
$h_{1,N}(\tau)=(\wp_{\left[\begin{smallmatrix}0\\1/N\end{smallmatrix}\right]}(\tau)
-\wp_{\left[\begin{smallmatrix}0\\1/2\end{smallmatrix}\right]}(\tau))/
(\wp_{\left[\begin{smallmatrix}0\\1/2\end{smallmatrix}\right]}(\tau)-
\wp_{\left[\begin{smallmatrix}0\\1/4\end{smallmatrix}\right]}(\tau))$.
\item[\textup{(iii)}] $\mathcal{O}_N=\mathcal{O}_{1,N}[h_N]$, where $h_N(\tau)
=(\wp_{\left[\begin{smallmatrix}1/N\\0\end{smallmatrix}\right]}(\tau)-
\wp_{\left[\begin{smallmatrix}0\\1/2\end{smallmatrix}\right]}(\tau))/
(\wp_{\left[\begin{smallmatrix}0\\1/2\end{smallmatrix}\right]}(\tau)-
\wp_{\left[\begin{smallmatrix}0\\1/4\end{smallmatrix}\right]}(\tau))$.
\end{itemize}
\end{theorem}
\begin{proof}
(i) Since $g_{1,4}$ and $g_{1,4}-16$ are modular units in $\mathbb{C}(X_1(4))$ by
Lemma \ref{generators}(iii) and (\ref{observation}), we get
the inclusion $\mathcal{O}_{1,4}\supseteq\mathbb{C}[g_{1,4},g_{1,4}^{-1},(g_{1,4}-16)^{-1}]$. \par
Conversely, let $h\in\mathcal{O}_{1,4}$. Then it is a rational function of $g_{1,4}$ by Lemma
\ref{generators}(iii), namely
$h=P(g_{1,4})/Q(g_{1,4})$ for some polynomials
$P(X),Q(X)\in\mathbb{C}[X]$ which are relatively prime. If $Q(X)$ has a linear factor other than $g_{1,4}$ and $g_{1,4}-16$,
then $h$ has a pole on $\mathbb{H}$ by (\ref{observation}). Hence we obtain the reverse inclusion
$\mathcal{O}_{1,4}\subseteq\mathbb{C}[g_{1,4},g_{1,4}^{-1},(g_{1,4}-16)^{-1}]$. This proves (i).\\
(ii) Since $h_{1,N}\in\mathcal{O}_{1,N}$ by Lemma \ref{generators}(i) and the paragraph below Lemma \ref{wp},
we have the inclusion $\mathcal{O}_{1,N}\supseteq\mathcal{O}_{1,4}[h_{1,N}]$.
\par
Let $f\in\mathcal{O}_{1,N}$. Since
\begin{eqnarray*}
\mathbb{C}(X_1(N))&=&\mathbb{C}(j,(g_2g_3/\Delta)\wp_{\left[\begin{smallmatrix}0\\1/N\end{smallmatrix}\right]})
\quad\textrm{by Lemma \ref{generators}(i)}\\
&=&\mathbb{C}(X_1(4))((g_2g_3/\Delta)\wp_{\left[\begin{smallmatrix}0\\1/N\end{smallmatrix}\right]})\quad
\textrm{because}~j\in\mathbb{C}(X_1(4))\\
&=&\mathbb{C}(X_1(4))((g_2g_3/\Delta)((\wp_{\left[\begin{smallmatrix}0\\1/2\end{smallmatrix}\right]}
-
\wp_{\left[\begin{smallmatrix}0\\1/4\end{smallmatrix}\right]})h_{1,N}+\wp_{\left[\begin{smallmatrix}0\\1/2\end{smallmatrix}\right]}))\\
&=&\mathbb{C}(X_1(4))(h_{1,N})\\&&
\textrm{because}
~(g_2g_3/\Delta)\wp_{\left[\begin{smallmatrix}0\\1/2\end{smallmatrix}\right]},
(g_2g_3/\Delta)\wp_{\left[\begin{smallmatrix}0\\1/4\end{smallmatrix}\right]}\in\mathbb{C}(X_1(4))~\textrm{by Lemma \ref{generators}(i)} ,
\end{eqnarray*}
$f$ can be written in the form
\begin{equation}\label{expression}
f=r_0+r_1h+r_2h^2+\cdots+r_{d-1}h^{d-1}
\end{equation}
where $h=h_{1,N}$, $d=[\mathbb{C}(X_1(N)):\mathbb{C}(X_1(4))]$ and $r_0,r_2,\ldots, r_{d-1}\in \mathbb{C}(X_1(4))$.
Multiplying both sides of the equation (\ref{expression}) by
$1,h,\ldots, h^{d-1}$, respectively, we attain a linear system (with unknowns $r_0,r_1,\ldots,r_{d-1}$)
\begin{equation*}
\left[\begin{matrix}1 & h & \cdots & h^{d-1}\\
h & h^2 & \cdots & h^{d}\\
\vdots & \vdots & \ddots & \vdots\\
h^{d-1} & h^d & \cdots & h^{2d-2}
\end{matrix}\right]
\left[\begin{matrix}
r_0 \\ r_1 \\ \vdots \\ r_{d-1}
\end{matrix}\right]
=\left[\begin{matrix}
f \\ hf \\ \vdots \\ h^{d-1}f
\end{matrix}\right].
\end{equation*}
Taking the trace $\mathrm{Tr}$ ($=\mathrm{Tr}_{\mathbb{C}(X_1(N))/\mathbb{C}(X_1(4))}$) on both sides
we achieve
\begin{equation}\label{system}
\left[\begin{matrix}\mathrm{Tr}(1) & \mathrm{Tr}(h) & \cdots & \mathrm{Tr}(h^{d-1})\\
\mathrm{Tr}(h) & \mathrm{Tr}(h^2) & \cdots & \mathrm{Tr}(h^d)\\
\vdots & \vdots & \ddots & \vdots\\
\mathrm{Tr}(h^{d-1}) & \mathrm{Tr}(h^d) & \cdots & \mathrm{Tr}(h^{2d-2})
\end{matrix}\right]
\left[\begin{matrix}
r_0 \\ r_1 \\ \vdots \\ r_{d-1}
\end{matrix}\right]
=\left[\begin{matrix}
\mathrm{Tr}(f)\\\mathrm{Tr}(hf) \\ \vdots \\ \mathrm{Tr}(h^{d-1}f)
\end{matrix}\right].
\end{equation}
Let $T$ be the $d\times d$ matrix in the left side of (\ref{system}), and let
$c_1,c_2,\ldots,c_d$ be the conjugates of $h\in\mathbb{C}(X_1(N))$ over $\mathbb{C}(X_1(4))$.
Then we get that
\begin{eqnarray*}
\det(T)&=&\left|\begin{array}{cccc}
\sum_{k=1}^d c_k^0 & \sum_{k=1}^d c_k^1 & \cdots & \sum_{k=1}^d c_k^{d-1}\\
\sum_{k=1}^d c_k^1 & \sum_{k=1}^d c_k^2 & \cdots & \sum_{k=1}^d c_k^{d}\\
\vdots & \vdots & \ddots & \vdots\\
\sum_{k=1}^d c_k^{d-1} & \sum_{k=1}^d c_k^d & \cdots & \sum_{k=1}^d c_k^{2d-2}\\
\end{array}\right|\\
&=&\left|\begin{array}{cccc}
c_1^0 & c_2^0 & \cdots & c_d^{0}\\
c_1^1 & c_2^1 & \cdots & c_d^{1}\\
\vdots & \vdots & \ddots & \vdots\\
c_1^{d-1} & c_2^{d-1} & \cdots & c_d^{d-1}
\end{array}\right|~
\left|\begin{array}{cccc}
c_1^0 & c_1^1 & \cdots & c_1^{d-1}\\
c_2^0 & c_2^1 & \cdots & c_2^{d-1}
\\
\vdots & \vdots & \ddots & \vdots\\
c_d^0 & c_d^1 & \cdots & c_d^{d-1}
\end{array}\right|\\
&=&\prod_{1\leq m<n\leq d}(c_m-c_n)^2\quad\textrm{by the
Van der Monde determinant formula}.
\end{eqnarray*}
On the other hand, any conjugate of $h\in\mathbb{C}(X_1(N))$ over $\mathbb{C}(X_1(4))$ is of the form
\begin{equation*}
(\wp_{\left[\begin{smallmatrix}a/N\\b/N\end{smallmatrix}\right]}(\tau)-
\wp_{\left[\begin{smallmatrix}0\\1/2\end{smallmatrix}\right]}(\tau))/
(\wp_{\left[\begin{smallmatrix}0\\1/2\end{smallmatrix}\right]}(\tau)-
\wp_{\left[\begin{smallmatrix}0\\1/4\end{smallmatrix}\right]}(\tau))\quad\textrm{for some}~
\left[\begin{matrix}a\\b\end{matrix}\right]\in\mathbb{Z}^2-N\mathbb{Z}^2
\end{equation*}
owing to the fact $\mathrm{Gal}(\mathbb{C}(X_1(N))/\mathbb{C}(X_1(4)))\simeq
\overline{\Gamma}_1(N)/\overline{\Gamma}_1(4)$, the transformation formula (\ref{wptransf}) and Lemma \ref{wp}.
Moreover, we see that the function
\begin{equation*}
(\wp_{\left[\begin{smallmatrix}a/N\\b/N\end{smallmatrix}\right]}(\tau)
-\wp_{\left[\begin{smallmatrix}c/N\\d/N\end{smallmatrix}\right]}(\tau))/
(\wp_{\left[\begin{smallmatrix}0\\1/2\end{smallmatrix}\right]}(\tau)-
\wp_{\left[\begin{smallmatrix}0\\1/4\end{smallmatrix}\right]}(\tau))
\end{equation*}
for $\left[\begin{matrix}a\\b\end{matrix}\right],
\left[\begin{matrix}c\\d\end{matrix}\right]\in\mathbb{Z}^2-N\mathbb{Z}^2$ with
$\left[\begin{matrix}a\\b\end{matrix}\right]\not\equiv
\pm\left[\begin{matrix}c\\d\end{matrix}\right]\pmod{N\mathbb{Z}^2}$
has no zeros and poles on $\mathbb{H}$ by Lemma \ref{wp}. This implies that $\det(T)$ becomes a modular unit
in $\mathbb{C}(X_1(4))$, in particular,
$\det(T)$ belongs to $\mathcal{O}_{1,4}^\times$.
It then follows that $r_0,r_1,\ldots,r_{d-1}\in\mathcal{O}_{1,4}$, and hence we deduce
the inclusion $\mathcal{O}_{1,N}\subseteq\mathcal{O}_{1,4}[h_{1,N}]$.
This completes the proof of (ii).\\
(iii) In like manner as in the proof of (ii) one can prove (iii).
\end{proof}

\begin{remark}\label{explicit}
Let
\begin{eqnarray*}
\theta_2(\tau)=\sum_{n\in\mathbb{Z}}e^{\pi
i(n+1/2)^2\tau},\quad
\theta_3(\tau)=\sum_{n\in\mathbb{Z}}e^{\pi
in^2\tau}\quad\textrm{and}\quad
\theta_4(\tau)=\sum_{n\in\mathbb{Z}}(-1)^ne^{\pi in^2\tau}
\end{eqnarray*}
be the classical Jacobi theta functions, and let
\begin{equation}\label{eta}
\eta(\tau)=q^{1/24}\prod_{n=1}^{\infty}(1-q^n)
\end{equation}
be the Dedekind-eta function. Then they satisfy the relations
\begin{eqnarray}\label{Jacobiidentity}
\theta_2(\tau)^4+\theta_4(\tau)^4=\theta_3(\tau)^4,
\end{eqnarray}
and
\begin{eqnarray}\label{theta-eta}
\theta_2(2\tau)=2\eta(4\tau)^2/\eta(2\tau)\quad\textrm{and}\quad
\theta_4(2\tau)=\eta(\tau)^2/\eta(2\tau),
\end{eqnarray}
due to Jacobi
\cite[pp.27--29]{BGHZ}.
Furthermore, we have
\begin{equation*}
g_{1,4}(\tau)=16\theta_3(2\tau)^4/\theta_2(2\tau)^4
\end{equation*}
as a modular unit with
$\mathrm{ord}_q(g_{1,4}\circ\left[\begin{smallmatrix}0&-1\\1&0\end{smallmatrix}\right])=0$
\cite[Table3 and Theorem 6.2]{K-S}.
Hence we derive that
\begin{eqnarray*}
g_{1,4}(\tau)-16&=&16\theta_3(2\tau)^4/\theta_2(2\tau)^4-16\\
&=&16\theta_4(2\tau)^4/\theta_2(2\tau)^4\quad\textrm{by (\ref{Jacobiidentity})}\\
&=&\eta(\tau)^8/\eta(4\tau)^8\quad\textrm{by (\ref{theta-eta})}\\
&=&q^{-1}\prod_{n=1}^{\infty}(1+q^n)^{-8}(1+q^{2n})^{-8}\quad\textrm{by the definition (\ref{eta})}.
\end{eqnarray*}
Therefore, $g_{1,4}-16$ is indeed a modular unit.
\end{remark}

\begin{corollary}\label{sum}
Every weakly holomorphic modular function can be expressed as a sum of modular units \textup{(}of higher level\textup{)}.
\end{corollary}
\begin{proof}
Let $h$ be a weakly holomorphic modular function of level $N$.
Since it belongs to $\mathcal{O}_{4N/\gcd(4,N)}$
by Proposition \ref{characterize}, $h$ can be
written as a sum of modular units of level $4N/\gcd(4,N)$ by Theorem \ref{closure}. This completes the proof.
\end{proof}

Let $k$ and $N$ ($\geq1$) be integers.  We denote the vector space of all weakly holomorphic modular forms of level $N$ and weight $k$
by $\mathcal{M}_k^{!}(\Gamma(N))$. Then we have a graded algebra
\begin{equation*}
\mathcal{M}^!(\Gamma(N))=\bigoplus_{k\in\mathbb{Z}}\mathcal{M}_k^{!}(\Gamma(N))
\end{equation*}
with respect to weight $k$.
\par
Now, define a Klein form
\begin{equation*}
\mathfrak{k}_{\left[\begin{smallmatrix}0\\1/2\end{smallmatrix}\right]}(\tau)=(1/2\pi i)
g_{\left[\begin{smallmatrix}0\\1/2\end{smallmatrix}\right]}(\tau)/\eta(\tau)^2
\end{equation*}
which belongs to $\mathcal{M}_{-1}^!(\Gamma(8))$ \cite[Chapter 3 Theorem 4.1]{K-L}. It has no zeros and poles on $\mathbb{H}$
by the expansion formulas (\ref{Siegel}) and (\ref{eta}).

\begin{theorem}
For $N\equiv0\pmod{8}$, we get
\begin{equation*}
\mathcal{M}^!(\Gamma(N))=\mathcal{O}_N[\mathfrak{k}_{\left[\begin{smallmatrix}0\\1/2\end{smallmatrix}\right]},
\mathfrak{k}_{\left[\begin{smallmatrix}0\\1/2\end{smallmatrix}\right]}^{-1}]=
\mathbb{C}[g_{1,4},g_{1,4}^{-1},(g_{1,4}-16)^{-1},h_{1,N},h_N,
\mathfrak{k}_{\left[\begin{smallmatrix}0\\1/2\end{smallmatrix}\right]},
\mathfrak{k}_{\left[\begin{smallmatrix}0\\1/2\end{smallmatrix}\right]}^{-1}]
\end{equation*}
where $g_{1,4}$, $h_{1,N}$ and $h_N$ are functions described in \textup{Theorem \ref{closure}}.
\end{theorem}
\begin{proof}
It is obvious that $\mathcal{M}_0^!(\Gamma(N))=\mathcal{O}_N$.
\par
If $k\neq0$, then the following linear map
\begin{eqnarray*}
\varphi~:~\mathcal{O}_N&\rightarrow&\mathcal{M}_k^!(\Gamma(N))\\
h&\mapsto&\mathfrak{k}_{\left[\begin{smallmatrix}0\\1/2\end{smallmatrix}\right]}^{-k}h
\end{eqnarray*}
is an isomorphism, because
$\mathfrak{k}_{\left[\begin{smallmatrix}0\\1/2\end{smallmatrix}\right]}^{-1}\in\mathcal{M}_{1}^!(\Gamma(8))$
and $\mathfrak{k}_{\left[\begin{smallmatrix}0\\1/2\end{smallmatrix}\right]}\in\mathcal{M}_{-1}^!(\Gamma(8))$.
Thus $\mathcal{M}_k^!(\Gamma(N))=\mathfrak{k}_{\left[\begin{smallmatrix}0\\1/2\end{smallmatrix}\right]}^{-k}\mathcal{O}_N$
as an $\mathcal{O}_N$-module. Therefore we attain from Theorem \ref{closure}
\begin{eqnarray*}
\mathcal{M}^!(\Gamma(N))&=&\bigoplus_{k\in\mathbb{Z}}
\mathfrak{k}_{\left[\begin{smallmatrix}0\\1/2\end{smallmatrix}\right]}^{-k}\mathcal{O}_N\\
&=&\mathcal{O}_N[\mathfrak{k}_{\left[\begin{smallmatrix}0\\1/2\end{smallmatrix}\right]},
\mathfrak{k}_{\left[\begin{smallmatrix}0\\1/2\end{smallmatrix}\right]}^{-1}]\\
&=&\mathbb{C}[g_{1,4},g_{1,4}^{-1},(g_{1,4}-16)^{-1},h_{1,N},h_N,
\mathfrak{k}_{\left[\begin{smallmatrix}0\\1/2\end{smallmatrix}\right]},
\mathfrak{k}_{\left[\begin{smallmatrix}0\\1/2\end{smallmatrix}\right]}^{-1}].
\end{eqnarray*}

\end{proof}

\section {Siegel modular functions}

In this section
we shall show that if $f(Z)$ is a Siegel modular function of degree $g$ ($\geq2$)
that has no zeros and poles on $\mathbb{H}_g^\mathrm{diag}$, then $f(Z)$
is a product of $g$ modular units.

\begin{lemma}\label{each}
Let $g$ and $N$ be two positive integers $\geq2$. If $f(Z)$ is a Siegel modular function of degree $g$ and level $N$, then
the function
\begin{equation*}
f(\mathrm{diag}(\tau_1,\tau_2,\ldots,\tau_g))\quad(\mathrm{diag}(\tau_1,\tau_2,\ldots,\tau_g)\in\mathbb{H}_g^\mathrm{diag}),
\end{equation*}
as a function of $\tau_k$ \textup{(}$k=1,2,\ldots,g$\textup{)},
is a meromorphic modular function of level $N$.
\end{lemma}
\begin{proof}
Let
\begin{equation*}
\gamma_k=\left[\begin{matrix}a_k&b_k\\c_k&d_k\end{matrix}\right]\in\mathrm{SL}_2(\mathbb{Z})\quad(k=1,2,\ldots,g),
\end{equation*}
and set
\begin{equation*}
\gamma=\left[\begin{matrix}A&B\\C&D\end{matrix}\right]=
\left[\begin{matrix}\mathrm{diag}(a_1,a_2,\ldots,a_g) & \mathrm{diag}(b_1,b_2,\ldots,b_g)\\
\mathrm{diag}(c_1,c_2,\ldots,c_g) & \mathrm{diag}(d_1,d_2,\ldots,d_g)\end{matrix}\right],
\end{equation*}
where $A,B,C,D$ are $g\times g$ block matrices. Then we derive that
\begin{eqnarray*}
{}^t\gamma J\gamma&=&\left[\begin{matrix}A&C\\B&D\end{matrix}\right]
\left[\begin{matrix}0&-I_g\\I_g&0\end{matrix}\right]
\left[\begin{matrix}A&B\\C&D\end{matrix}\right]\quad\textrm{because $A,B,C,D$ are diagonal}\\
&=&\left[\begin{matrix}CA-AC&CB-AD\\DA-BC&DB-BD\end{matrix}\right]\\
&=&\left[\begin{matrix}0&\mathrm{diag}(c_1b_1-a_1d_1,\ldots,c_gb_g-a_gd_g)\\
\mathrm{diag}(d_1a_1-b_1c_1,\ldots,d_ga_g-b_gc_g)&0\end{matrix}\right]\\
&=&J\quad\textrm{due to}~\det(\gamma_k)=a_kd_k-b_kc_k=1~(k=1,2,\ldots,g),
\end{eqnarray*}
from which we see that $\gamma$ belongs to the group $\mathrm{Sp}_g(\mathbb{Z})$.
\par
And,
for $Z=\mathrm{diag}(\tau_1,\tau_2,\ldots,\tau_g)\in\mathbb{H}_g^\mathrm{diag}$
we achieve that
\begin{eqnarray}
\gamma(Z)&=&(AZ+B)
(CZ+D)^{-1}\nonumber\\
&=&\mathrm{diag}(a_1\tau_1+b_1,\ldots,a_g\tau_g+b_g)
\mathrm{diag}(c_1\tau_1+d_1,\ldots,c_g\tau_g+d_g)^{-1}\nonumber\\
&=&\mathrm{diag}((a_1\tau_1+b_1)(c_1\tau_1+d_1)^{-1},\ldots,(a_g\tau_g+b_g)
(c_g\tau_g+d_g)^{-1})\nonumber\\
&=&\mathrm{diag}(\gamma_1(\tau_1),\gamma_2(\tau_2),\ldots,\gamma_g(\tau_g)).\label{composition}
\end{eqnarray}
\par
On the other hand, assume that $\gamma_k\equiv I_2\pmod{N}$ for all $k=1,2,\ldots,g$.
Then $\gamma\equiv I_{2g}\pmod{N}$,
and for $Z=\mathrm{diag}(\tau_1,\tau_2,\ldots,\tau_g)\in\mathbb{H}_g^\mathrm{diag}$ we have
\begin{eqnarray*}
f(\mathrm{diag}(\tau_1,\tau_2,\ldots,\tau_g))&=&
f(Z)\\
&=&f(\gamma(Z))\quad\textrm{since $f$ is of level $N$}\\
&=&f(\mathrm{diag}(\gamma_1(\tau_1),\gamma_2(\tau_2),\ldots,\gamma_g(\tau_g)))\quad\textrm{by (\ref{composition})}.
\end{eqnarray*}
In particular, when $k$ is fixed ($k=1,2,\ldots,g$) and $\gamma_n=I_2$ for all $n\neq k$, we conclude
that $f(Z)$, as a function of $\tau_k$, is a meromorphic modular function of level $N$.
\end{proof}

\begin{theorem}\label{prod2}
Let $g$ and $N$ be two positive integers $\geq2$, and let
 $f(Z)$ be a Siegel modular function of degree $g$ and level $N$.
Then, $f(Z)$ has no zeros and poles on $\mathbb{H}_g^\mathrm{diag}$ if and only if there exist modular units
$v_1(\tau),v_2(\tau),\ldots,v_g(\tau)\in V_N$ such that
\begin{equation*}
f(\mathrm{diag}(\tau_1,\tau_2,\ldots,\tau_g))=
\prod_{k=1}^g v_k(\tau_k).
\end{equation*}
\end{theorem}
\begin{proof}
The proof of ``if" part is clear.\\
Conversely, assume that $f(Z)$ has no zeros and poles on $\mathbb{H}_g^\mathrm{diag}$.
Let $n$ ($\geq2$) be the number of inequivalent cusps
of $X(N)$. Since $V_N/\mathbb{C}^\times$ is a free abelian group of rank $n-1$ by Proposition \ref{rank},
there exist $g_1(\tau),g_2(\tau),\ldots,g_{n-1}(\tau)\in V_N$ such that
$V_N=\langle\mathbb{C}^\times,g_1,g_2,\ldots,g_{n-1}\rangle$.
Thus $f(\mathrm{diag}(\tau_1,\tau_2,\ldots,\tau_g))$, as a function of $\tau_g$, can be written as
\begin{equation}\label{prodf}
f(\mathrm{diag}(\tau_1,\tau_2,\ldots,\tau_g))=c(\tau_1,\tau_2,\ldots,\tau_{g-1})
\prod_{t=1}^{n-1}g_t(\tau_g)^{m_t(\tau_1,\tau_2,\ldots,\tau_{g-1})}
\end{equation}
where $c:\mathbb{H}^{g-1}\rightarrow\mathbb{C}^\times$ and $m_t:\mathbb{H}^{g-1}\rightarrow\mathbb{Z}$ are functions of $\tau_1,\tau_2,
\ldots,\tau_{g-1}$ by Lemma \ref{each} and the assumption.
\par
Then we deduce that
\begin{eqnarray*}
&&\prod_{\gamma\in\overline{\Gamma}(1)/\overline{\Gamma}(N)}
f(\mathrm{diag}(\tau_1,\tau_2,\ldots,\tau_{g-1},\gamma(\tau_g)))\\
&=&\prod_{\gamma\in\overline{\Gamma}(1)/\overline{\Gamma}(N)}\bigg(
c(\tau_1,\tau_2,\ldots,\tau_{g-1})
\prod_{t=1}^{n-1}g_t(\gamma(\tau_g))^{m_t(\tau_1,\tau_2,\ldots,\tau_{g-1})}
\bigg)\quad\textrm{by (\ref{prodf})}\\
&=&c(\tau_1,\tau_2,\ldots,\tau_{g-1})^d\prod_{t=1}^{n-1}\bigg(\prod_{\gamma\in\overline{\Gamma}(1)/\overline{\Gamma}(N)}
g_t(\gamma(\tau_g))
\bigg)^{m_t(\tau_1,\tau_2,\ldots,\tau_{g-1})},
\quad\textrm{where}~d=|\overline{\Gamma}(1)/\overline{\Gamma}(N)|\\
&=&c(\tau_1,\tau_2,\ldots,\tau_{g-1})^d\prod_{t=1}^{n-1}\mathrm{N}_{\mathbb{C}(X(N))/\mathbb{C}(X(1))}
(g_t(\tau_g))^{m_t(\tau_1,\tau_2,\ldots,\tau_{g-1})}\\
&&\textrm{due to the fact}~\mathrm{Gal}(\mathbb{C}(X(N))/\mathbb{C}(X(1)))\simeq\overline{\Gamma}(1)/\overline{\Gamma}(N)\\
&=&c(\tau_1,\tau_2,\ldots,\tau_{g-1})^d\prod_{t=1}^{n-1}c_t^{m_t(\tau_1,\tau_2,
\ldots,\tau_{g-1})}\quad\textrm{for some}~c_1,c_2,\cdots,c_{n-1}\in\mathbb{C}^\times~\textrm{by Remark \ref{V_1}},
\end{eqnarray*}
which is a modular unit of level $N$ as a function of each $\tau_k$ ($k=1,2,\cdots,g-1$) by
Lemma \ref{each}.
It follows from (\ref{prodf}) that
\begin{eqnarray*}
&&f(\mathrm{diag}(\tau_1,\tau_2,\ldots,\tau_g))^d/\prod_{\gamma\in\overline{\Gamma}(1)/\overline{\Gamma}(N)}
f(\mathrm{diag}(\tau_1,\tau_2,\ldots,\tau_{g-1},\gamma(\tau_g)))\\
&=&
\bigg(c(\tau_1,\tau_2,\ldots,\tau_{g-1})\prod_{t=1}^{n-1}g_t(\tau_g)^{m_t(\tau_1,\tau_2,\ldots,\tau_{g-1})}\bigg)^d/
\bigg(c(\tau_1,\tau_2,\ldots,\tau_{g-1})^d\prod_{t=1}^{n-1}c_t^{m_t(\tau_1,\tau_2,\ldots,\tau_{g-1})}\bigg)\\
&=&\prod_{t=1}^{n-1}(c_t^{-1}g_t(\tau_g)^d)^{m_t(\tau_1,\tau_2,\ldots,\tau_{g-1})}.\label{divide}
\end{eqnarray*}
Now, set this function to be $h(\tau_1,\tau_2,\ldots,\tau_g)$ which is a modular unit as a function of each $\tau_k$ ($k=1,2,\ldots,g$).
\par
On the other hand, when $\tau_g\in\mathbb{H}$ is fixed,
the image of the holomorphic function
\begin{eqnarray}
\varphi~:~\mathbb{H}^{g-1}&\rightarrow&\mathbb{C}^\times\label{original}\\
(\tau_1,\tau_2,\ldots,\tau_{g-1})&\mapsto&h(\tau_1,\tau_2,\ldots,\tau_g)=
\prod_{t=1}^{n-1}(c_t^{-1}g_t(\tau_g)^d)^{m_t(\tau_1,\tau_2,\ldots,\tau_{g-1})}\nonumber
\end{eqnarray}
is a countable set, because $m_t(\tau_1,\tau_2,\ldots,\tau_{g-1})$ ($t=1,2,\ldots,n-1$) are integer-valued functions.
Let $\ell$ be an index in $\{1,2,\ldots,g-1\}$ and
suppose that $\tau_1,\tau_2,\ldots,\tau_{g-1}$ are fixed except for $\tau_\ell$. Then $\varphi$
can be viewed as a holomorphic map from $\mathbb{H}$ to $\mathbb{C}^\times$
with respect to $\tau_\ell$. Since its image is a countable set as mentioned above,
the modular unit $h(\tau_1,\tau_2,\ldots,\tau_g)$, as a function of $\tau_\ell$,
must be a constant by Lemma \ref{constsurj}.
This observation essentially indicates that the map $\varphi$
defined on $\mathbb{H}^{g-1}$ in (\ref{original}) is in fact a
constant, and hence
the function $h(\tau_1,\tau_2,\ldots,\tau_g)$ of $g$ variables is a function of $\tau_g$.
Moreover, since $g_1(\tau),g_2(\tau),\ldots,g_{n-1}(\tau)$ form a basis for
the free abelian group $V_N/\mathbb{C}^\times$,
the integer-valued functions $m_t(\tau_1,\tau_2,\ldots,\tau_{g-1})$ $(t=1,2,\ldots,n-1 $) should be fixed integers, say $m_t$.
Thus if we set $v_g(\tau)=\prod_{t=1}^{n-1}g_t(\tau)^{m_t}\in V_N$, then we derive from (\ref{prodf}) that
\begin{equation}\label{deduce}
f(\mathrm{diag}(\tau_1,\tau_2,\ldots,\tau_g))=c(\tau_1,\tau_2,\ldots,\tau_{g-1})v_g(\tau_g).
\end{equation}
\par
The only property of
$f(\mathrm{diag}(\tau_1,\tau_2,\ldots,\tau_g))$
necessary to have (\ref{deduce}) is that
it is a
meromorphic modular function of level $N$
as a function of each $\tau_k$ ($k=1,2,\ldots,g$). Now that
$c(\tau_1,\tau_2,\ldots,\tau_{g-1})$ retains this property,
if we apply the same argument to $c(\tau_1,\tau_2,\ldots,\tau_{g-1})$ instead of $f(\mathrm{diag}(\tau_1,\tau_2,\ldots,\tau_g))$
and repeat this process over and over again, then we eventually reach the conclusion after $(g-1)$ steps.
\end{proof}

\begin{example}\label{thetaconstant}
Let $g,N\geq1$.
For $\mathbf{r}=\left[\begin{matrix}r_1\\\vdots\\r_g\end{matrix}\right],
\mathbf{s}=\left[\begin{matrix}s_1\\\vdots\\s_g\end{matrix}\right]\in\mathbb{Q}^g$ we
define a theta constant by
\begin{equation*}
\Theta_{\left[\begin{smallmatrix}\mathbf{r}\\\mathbf{s}\end{smallmatrix}\right]}(Z)=
\sum_{\mathbf{n}\in\mathbb{Z}^g}
e({}^t(\mathbf{n}+\mathbf{r})Z(\mathbf{n}+\mathbf{r})/2+
{}^t(\mathbf{n}+\mathbf{r})\mathbf{s})\quad(Z\in\mathbb{H}_g),
\end{equation*}
where $e(z)=e^{2\pi iz}$ for $z\in\mathbb{C}$. We further set
\begin{equation*}
\Phi_{\left[\begin{smallmatrix}\mathbf{r}\\\mathbf{s}\end{smallmatrix}\right]}(Z)=
\Theta_{\left[\begin{smallmatrix}\mathbf{r}\\\mathbf{s}\end{smallmatrix}\right]}(Z)/
\Theta_{\left[\begin{smallmatrix}\mathbf{0}\\\mathbf{0}\end{smallmatrix}\right]}(Z)\quad(Z\in\mathbb{H}_g),
\end{equation*}
which is a Siegel modular function of level $2N^2$ \cite[Proposition 7]{Shimura2}.
\par
Now, we assume that $g\geq2$, $Z'\in\mathbb{H}_{g-1}$ and $\tau\in\mathbb{H}$.
Then we derive that
\begin{eqnarray}
&&\Theta_{\left[\begin{smallmatrix}\mathbf{r}\\\mathbf{s}\end{smallmatrix}\right]}
\bigg(\left[\begin{matrix}Z'&0\\0&\tau\end{matrix}\right]\bigg)\nonumber\\
&=&
\sum_{\mathbf{n}=\left[\begin{smallmatrix}n_1\\\vdots\\n_g\end{smallmatrix}\right]\in\mathbb{Z}^g}
e\bigg(\frac{1}{2}{^t}\left[\begin{matrix}\mathbf{n}'+\mathbf{r}'\\n_g+r_g\end{matrix}\right]
\left[\begin{matrix}Z'&0\\0&z\end{matrix}\right]
\left[\begin{matrix}\mathbf{n}'+\mathbf{r}'\\n_g+r_g\end{matrix}\right]
+{^t}\left[\begin{matrix}\mathbf{n}'+\mathbf{r}'\\n_g+r_g\end{matrix}\right]
\left[\begin{matrix}\mathbf{s}'\\s_g\end{matrix}\right]\bigg),
\quad\textrm{where}~\mathbf{n}'=\left[\begin{matrix}n_1\\\vdots\\n_{g-1}\end{matrix}\right]\nonumber\\
&=&\sum_{\mathbf{n}'\in\mathbb{Z}^{g-1}}\sum_{n_g\in\mathbb{Z}}
e(^{t}(\mathbf{n}'+\mathbf{r}')Z'(\mathbf{r}'+\mathbf{s}')/2+(n_g+r_g)\tau(n_g+r_g)/2
+^{t}(\mathbf{n}'+\mathbf{r}')\mathbf{s}'+(n_g+r_g)s_g)\nonumber\\
&=&
\bigg(\sum_{\mathbf{n}'\in\mathbb{Z}^{g-1}}
e(^{t}(\mathbf{n}'+\mathbf{r}')Z'(\mathbf{r}'+\mathbf{s}')/2
+^{t}(\mathbf{n}'+\mathbf{r}')\mathbf{s}')\bigg)
\bigg(\sum_{n_g\in\mathbb{Z}}
e((n_g+r_g)\tau(n_g+r_g)/2+(n_g+r_g)s_g)\bigg)\nonumber\\
&=&\Theta_{\left[\begin{smallmatrix}\mathbf{r}'\\\mathbf{s}'\end{smallmatrix}\right]}(Z')
\Theta_{\left[\begin{smallmatrix}r_g\\s_g\end{smallmatrix}\right]}(\tau).\label{induc}
\end{eqnarray}
Applying this argument inductively we obtain
\begin{equation*}
\Phi_{\left[\begin{smallmatrix}\mathbf{r}\\\mathbf{s}\end{smallmatrix}\right]}(
\mathrm{diag}(\tau_1,\tau_2,\ldots,\tau_g))=\prod_{k=1}^g
\Phi_{\left[\begin{smallmatrix}r_k\\s_k\end{smallmatrix}\right]}(\tau_k)
\quad(\mathrm{diag}(\tau_1,\tau_2,\ldots,\tau_g)\in\mathbb{H}_g^\mathrm{diag}).
\end{equation*}
\par
On the other hand,
it follows from the Jacobi triple product identity \cite[(17.3)]{Fine} and the definition (\ref{Siegel}) in $\S$\ref{sectwo} that
\begin{equation*}
\Phi_{\left[\begin{smallmatrix}r\\s\end{smallmatrix}\right]}(\tau)
=\left\{
\begin{array}{ll}
e((2rs+r-s)/4)g_{\left[\begin{smallmatrix}1/2-r\\1/2-s\end{smallmatrix}\right]}(\tau)/
g_{\left[\begin{smallmatrix}1/2\\1/2\end{smallmatrix}\right]}(\tau) & \textrm{if}~
\left[\begin{smallmatrix}r\\s\end{smallmatrix}\right]\in\mathbb{Q}^2-(1/2+\mathbb{Z})^2,\\
0 & \textrm{if}~\left[\begin{smallmatrix}r\\s\end{smallmatrix}\right]\in(1/2+\mathbb{Z})^2.
\end{array}\right.
\end{equation*}
Therefore we conclude that
$\Phi_{\left[\begin{smallmatrix}\mathbf{r}\\\mathbf{s}\end{smallmatrix}\right]}(
\mathrm{diag}(\tau_1,\tau_2,\ldots,\tau_g))$ has no zeros and poles on $\mathbb{H}_g^\mathrm{diag}$, or is identically zero.
\end{example}

\bibliographystyle{amsplain}

\address{
Department of Mathematical Sciences \\
KAIST \\
Daejeon 305-701 \\
Republic of Korea} {zandc@kaist.ac.kr\\jkkoo@math.kaist.ac.kr}
\address{
Department of Mathematics\\
Hankuk University of Foreign Studies\\
Yongin-si, Gyeonggi-do 449-791\\
Republic of Korea } {dhshin@hufs.ac.kr}

\end{document}